\documentclass[amstex,12pt]{amsart}
\usepackage{geometry}
\usepackage{mathtools,amssymb,amsthm,mathrsfs,color,lineno,paralist,graphicx,float}
\usepackage[linktocpage,backref=page,colorlinks,
linkcolor=blue,
anchorcolor=blue,
citecolor=blue,
]{hyperref}
\usepackage{color}
\usepackage[T1]{fontenc}
\usepackage[utf8]{inputenc}
\usepackage{mathrsfs}
\usepackage{cleveref}
\usepackage{mathscinet}
\usepackage{calc}
\usepackage{etoolbox}

\newtheorem{thm}{Theorem}[section]
\theoremstyle{plain}
\theoremstyle{definition}
\newtheorem{proposition}[thm]{Proposition}

\newtheorem{lemma}[thm]{Lemma}
\newtheorem{definition}[thm]{Definition}
\newtheorem{claim}[thm]{Claim}

\newtheorem{rem}[thm]{Remark}

\numberwithin{equation}{section}

\definecolor{lgray}{gray}{0.9}

\def\leq{\leqslant}
\def\geq{\geqslant}

\def\hat{\widehat}

\allowdisplaybreaks
\renewcommand*{\backref}[1]{}
\renewcommand*{\backrefalt}[4]{\quad \tiny
  \ifcase #1 (\textbf{NOT CITED.})
  \or    (Cited on page~#2.)
  \else   (Cited on page~#2.)
  \fi}
\makeatletter

\title[Observable full-horseshoes]{Observable full-horseshoes for Lagrangian flows advected by stochastic 2D Navier-Stokes equations}

\author{Wen Huang}
\address[Wen Huang]{School of Mathematical Sciences\\ University of Science and Technology of China\\ Hefei, Anhui, 230026, China}
\email{wenh@mail.ustc.edu.cn}

\author{Jianhua Zhang}
\address[Jianhua Zhang]{School of Mathematical Sciences\\ University of Science and Technology of China\\ Hefei, Anhui, 230026, China}
\email{leapforg@mail.ustc.edu.cn}

\date{\today}

\begin{document}

\begin{abstract}
In this paper, we mainly study the turbulence  of  Lagrangian flow advected by  stochastic 2D Navier-Stokes equations.  It is proved that this system has observable full-horseshoes. The observable  full-horseshoe means that  it is a  kind of chaotic structure and   occurs on  any two  disjoint non-empty closed balls.   
\end{abstract}
\maketitle


\section{Introduction}
\label{23-8-6-1544}

Turbulence   is  said that   a fluid motion becomes complicated, unpredictable, irregular and chaotic over time.  This kind of  physical phenomenon ubiquitously exists in nature (for example, see \cite{majda2006nonlinear,pope2000turbulent,salmon1998lectures,vallis2017atmospheric} ).  There are amount  of  experimental  results to indicate that some turbulent fluids systems  are  sensitive dependence on initial conditions, such as \cite{MR0284067,GS,Lib}.   Recently, Bedrossian  et.al  made a breakthrough.   And  they   proved  some kinds of fluid models including Lagrangian flow advected by  stochastic 2D Navier-Stokes equation and Galerkin truncation of stochastic 2D Navier-Stokes equations, have positive Lyapunov exponents in \cite{MR4404792,MR4372219,bedrossian2021chaos}.  

In  chaotic dynamical systems, another kind of landmark of chaos is  horseshoe, which was introduced by Smale in \cite{MR0182020}.    It is a powerful geometry tool to describe the complex behaviour of the systems, for example see \cite{MR2140094}. It is natural to ask whether there is  horseshoe or horseshoe-like structure for turbulent dynamical system.  Lately, the authors have proved the Galerkin truncation of stochastic 2D Navier-Stokes equations have full-horseshoes  which is a weaker chaotic structure than horseshoes in \cite{axXiv: 230305027}.    In this paper, we  devote  to describe turbulence  of  Lagrangian flows advected by stochastic 2D Navier-Stokes equations by using full-horseshoes. Despite the statistical property of random  Lagrangian flow has extensively been investigated a lot (for example, see \cite{CFVP, Zbl 1205.76133, Zbl 0796.76084, Zbl 1273.76083, Zbl 0193.27106,MR4264955}),   few literatures  characterize the chaos of  random Lagrangian flow from   perspective of sample pathwise.  In this paper,  we obtain that the random Lagrangian flow  has full-horseshoes on any two disjoint non-empty closed balls for almost sure sample and initial   velocity (see  \Cref{23-7-4-2042}). The  crucial point of our result is that  it provides  a possibility to capture the chaotic behaviour of  random Lagrangian flow.

\subsection{Random Lagrangian flow}
In this subsection, we introduce  random Lagrangian flow and some basic assumptions.   Denote $\mathbb{T}^2$ as  the two-dimensional torus.  The Lagrangian flow advected by   stochastic  2D Navier-Stokes equations   is  diffeomorphisms $\varphi_t(\omega, u): \mathbb{T}^2\to\mathbb{T}^2$ for $t\geq0$ defined by following random ordinary differential 
equation 
\begin{align}
\label{23-8-4-1509}
\frac{\mathrm{d}\varphi_t(\omega, u)x}{\mathrm{d}t}=u_t(\omega,u)( \varphi_t(\omega, u)x),\quad \varphi_0(\omega, u)x=x.
\end{align}
  Here, the random velocity field $u_t: \Omega\times\mathbb{H}\times\mathbb{T}^2\to\mathbb{R}^2$ is  the solution of  following 2D   incompressible stochastic Navier-Stokes equation with $u_0(\cdot,u)\equiv u$,
\begin{align}
\label{23-7-4-1717}
\partial_tu_t+(u_t\cdot \nabla) u_t=\epsilon\Delta u_t-\nabla p +\dot{W}_t,\quad \text{div}(u_t)=0,
\end{align}
where $\epsilon$ is the fixed positive constant viscosity, $p$ is the pressure,  and  $\dot{W}_t$ is the stochastic external force  described more precisely below. 

Take the phase space of \Cref{23-7-4-1717} as following  Hilbert space
 $$\mathbb{H}=\{u\in \mathbf{H}^{s}(\mathbb{T}^2;\mathbb{R}^2): \int_{\mathbb{T}^2}u\mathrm{d}x=\boldsymbol{0}, \text{div}(u)=0\},$$
where $s\geq 4$.
Define a basic  of $\mathbb{H}$ by 
\begin{align*}
e_{k}(x)=
\begin{cases}
\frac{k^{\perp}}{|k|}\sin (k\cdot x)\quad & k\in\mathbb{Z}^2_+,
\\\frac{k^{\perp}}{|k|}\cos (k\cdot x)\quad & k\in\mathbb{Z}^2_- ,
\end{cases}
\end{align*}
where $k^{\perp}=(k_2, -k_1)$,  
$$\mathbb{Z}^2_+=\{(k_1, k_2)\in \mathbb{Z}^2: k_2>0\}\cup\{(k_1, k_2)\in \mathbb{Z}^2: k_1>0, k_2=0\}\text{ and }\mathbb{Z}_-^2=-\mathbb{Z}^2_+.$$  
Denote $(\Omega,\mathscr{F},\mathbb{P})$ as an infinite-dimensional  Wiener space, i.e. 
\begin{align}
\label{23-8-28-1017}
(\Omega,\mathscr{F},\mathbb{P})=\big(C_0([0,+\infty), \mathbb{R}),\mathscr{F}_0,\mathbb{P}_0\big)^{\mathbb{Z}^2_0},
\end{align}
where $\mathscr{F}_0$ is the Borel $\sigma$-algebra of $C_0([0,+\infty), \mathbb{R})$ with compact open topology, $\mathbb{P}_0$ is the Wiener measure on $(C_0([0,+\infty), \mathbb{R}),\mathscr{F}_0)$, and $\mathbb{Z}_0^2=\mathbb{Z}_+^2\cup\mathbb{Z}_-^2$.  Then $\{W^k_t(\omega):=\omega_k(t)\}_{k\in\mathbb{Z}^2_0}$ is  a family of independent   one-dimensional Wiener processes on $(\Omega,\mathscr{F},\mathbb{P})$.  Throughout  this paper, we will consider a white-in-time stochastic forcing $\dot{W}_t$ being the form   
$$\dot{W}_t=\sum_{k\in\mathbb{Z}_0^2}q_ke_k\dot{W}_t^k, $$
where  $q_k$ are non-negative constants with following assumptions   
\begin{enumerate}[(A1)]
\item [(Low mode non-degeneracy):]  $q_k>0$ whenever $k=(\pm1,0)$ or $(0,\pm1)$;
\item [(High mode non-degeneracy):] there exist some   $\alpha\in (s+1, s+2)$ and some large positive integer $L$ such that 
\begin{align*}
q_k\approx |k|^{-\alpha},
\end{align*}
holds for any  $k\in\{ k\in\mathbb{Z}_0^2: \max\{|k_1|, |k_2|\}\geq L\}$.
\end{enumerate}

The above two assumptions mainly ensure that the Lagrangian flow has positive Lyapunov exponents. The reader can refer  to \cite[Theorem 1.6]{MR4404792} for details.   
Now,  we summary the well-posedness and properties of velocity flow $(u_t)_{t\geq0}$ and Largrangian  flow $(\varphi_t)_{t\geq0}$.  
\begin{proposition}
\label{23-9-2-2037}
Under above setting, one has that 
\begin{enumerate}
\item For  all $u\in\mathbb{H}$, \Cref{23-7-4-1717}  exists a unique mild solution $(u_t(\omega,u))_{t\geq0}\in C([0, \infty); \mathbb{H})$ with $u_0(\omega,u)=u$ for almost trajectory. And the Markov process $(u_t)_{t\geq0}$ admits a unique Borel stationary measure $\mu$ in $\mathbb{H}$.
\item   \Cref{23-8-4-1509} exists a unique  Carath\'eodory solution 
$$(\varphi_t(\omega,u))_{t\geq0}\in C([0, \infty); \text{Diff}^2(\mathbb{T}^2))$$ with $\varphi_0(\omega,u)= \text{Id}_{\mathbb{T}^2}$ for almost trajectory. And the Markov process $(u_t,\varphi_t)_{t\geq0}$ admits a unique Borel stationary measure $\mu\times\text{vol}$ in $\mathbb{H}\times\mathbb{T}^2$, where $\text{vol}$ is the volume measure on $\mathbb{T}^2$.
\end{enumerate}
\end{proposition} 
 The existence and uniqueness of  solution of \Cref{23-7-4-1717}  follows from \cite[Chapter 15]{MR1417491} or \cite[Chapter 2]{MR3443633}; the unique ergodicity can be obtained by adjusting  the arguments in   \cite{MR1346374, MR2770903}.  The existence of uniqueness of solution of \Cref{23-8-4-1509} follows from the regularity of $u_t$ and \cite[Chapter 2]{A}; the unique ergodicity was proved in \cite[Section 7]{MR4404792}.

\subsection{Main result and discussions}
The existence of full-horseshoes for Lagrangian flows can be obtained  by checking  conditions  in  \cite[Proposition 5.6]{axXiv: 230305027}.  But we  give more precise argument for  the Lagrangian flow to determine the location of the full-horseshoes in this paper. Particularly, we give an abstract result to ensure the existence of observable full-horseshoes (see \Cref{23-8-19-2012}).  Now,  we state  the main result of this paper as follows. 
 \begin{thm}
 \label{23-7-4-2042}
 The Lagrangian flow $(\varphi_t)_{t\geq0}$ defined by \Cref{23-8-4-1509} has full-horseshoes on any two  disjoint non-empty closed balls of $\mathbb{T}^2$. Namely,  given any two disjoint non-empty closed balls $\{U_1, U_2\}$ of $\mathbb{T}^2$ , for $\mathbb{P}\times\mu$-a.s. $(\omega, u)\in\Omega\times\mathbb{H}$, there  exists a  subset $J(\omega,u)$ of $\mathbb{Z}_+:=\mathbb{N}\cup\{0\}$ such that 
\begin{enumerate}[(a)]
\item   $\lim_{n\to+\infty}\frac{|J(\omega, u)\cap\{0,1,\dots, n-1\}|}{n}>0$;
\item  for any $s\in\{1,2\}^{J(\omega, u)}$, there exists  an $x_s\in \mathbb{T}^2$ such that $\varphi_j(\omega,u)x_s\in U_{s(j)}$ for each $j\in J(\omega,u)$.
\end{enumerate}
 \end{thm}
 
\begin{rem}
In fact, the Lagrangian flow has observable full-horseshoes of  any discrete time form and the hitting time has uniformly positive lower bound. Particularly,   for any  $t>0$ and any two disjoint non-empty closed balls $\{U_1, U_2\}$ of $\mathbb{T}^2$,  there exists positive constants  $b$  such that   for $\mathbb{P}\times\mu$-a.s. $(\omega, u)\in\Omega\times\mathbb{H}$, there  is a  subset $J(\omega,u)$ of $\mathbb{Z}_+$ such that 
\begin{enumerate}[(a)]
\item   $\lim_{n\to+\infty}\frac{|J(\omega, u)\cap\{0,1,\dots, n-1\}|}{n}>b$;
\item  for any $s\in\{1,2\}^{J(\omega, u)}$, there exists  an $x_s\in \mathbb{T}^2$ such that $\varphi_{tj}(\omega,u)x_s\in U_{s(j)}$ for each $j\in J(\omega,u)$.
\end{enumerate}

\end{rem}

Recently,  the  phenomenon of observable full-horseshoe has been obtained  in \cite{arXiv:2304.03685}  for a kind of  one-dimensional expanding  random dynamical system.  However, the  method in \cite{arXiv:2304.03685}  looks like to be invalid  for high-dimensional dynamical system.   In this paper, we adopt   totally different method, which is based on the hyperbolic property of stationary measure and authors' recent work \cite{axXiv: 230305027}.  We believe our method can relax the condition of stationary measure to ensure the existence of observable full-horseshoes.

This paper is organized as following: In \Cref{22-12-05-01}, we mainly review the basic knowledge of  entropy,  Pinsker $\sigma$-algebra,  and $K$-system for random dynamical system.  In \Cref{23-8-13-2002},  borrowing the invariant manifold theory in  smooth random dynamical systems,   we  give a sufficient condition of a system  to guarantee that   it is a $K$-system.   In \Cref{23-8-22-1709},  we mainly verify that  Lagrangian flow is a $K$-system  and use the result in \cite{axXiv: 230305027} to complete the proof of   \Cref{23-7-4-2042}.

 \subsection{Acknowledgments}
The authors were supported by NSFC of China (12090012, 12090010, 12031019,11731003). The second author would like to thanks 
College of Mathematics  of Sichuan University for their warm hospitality,  as the idea of this part work was  formulated when he was visiting  here.

 \section{Entropy, Pinsker  $\sigma$-algebra and $K$-system for RDS}
\label{22-12-05-01}


 In this section, we review some basic concepts and classical results  about  measurable partition, entropy,   Pinsker $\sigma$-algebra and $K$-system for random dynamical system. The reader can see \cite{A, EW,G,Wal} for details.  

\subsection{Measurable partition  and  RDS}
\label{22-07-06-01}
In this section,  we always assume that triple  $(X, \mathscr{B}_X,\mu)$  is  a \emph{Polish probability space}, which means that $X$ is a Polish space, $\mathscr{B}_X$ is the Borel $\sigma$-algebra of $X$,  $\mu$ is a  probability measure on  $(X, \mathscr{B}_X)$.

\begin{definition}[Measurable partition]
A partition $\alpha$ of $(X, \mathscr{B}_X,\mu)$ is called the  \emph{measurable partition} if there exists a family of countable measurable subsets $\{A_i\}_{i\in\mathbb{N}}$  satisfying that (1)  for every $i\in\mathbb{N}$ $A_i$ is the union of elements in $\alpha$,  (2) for any two distinct elements $B_1,B_2$ of $\alpha$ 
there exists some $i\in\mathbb{N}$ such that either $B_1\subset A_i, B_2\not\subset A_i$ or $B_1\not\subset A_i, B_2\subset A_i$. Obviously, the elements in $\alpha$ are measurable.
   \end{definition}
 For any measurable partition $\alpha$ of $(X, \mathscr{B}_X,\mu)$,  denote  $\hat{\alpha}$ as  the $\sigma$-algebra generated by measurable set in $\mathscr{B}_X$ which is  the union  of elements in $\alpha$.  Then  we define the  inclusion relationship of measurable partition through the inclusion relationship of $\sigma$-algebra. i.e.
\begin{enumerate}
\item given any two partitions $\alpha_1$ and $\alpha_2$ of $(X,\mathscr{B}_X,\mu)$,  we say that $\alpha_1$  is finer (coarser) than $\alpha_2$, and   denote it by  $\alpha_2\prec \alpha_1$ ($\alpha_2\succ\alpha_1$)   if  $\sigma$-algebra $\hat{\alpha}_2\subset\hat{\alpha}_1\pmod\mu$ ($\hat{\alpha}_2\supset\hat{\alpha}_1\pmod\mu$). Particularly, denote $\alpha_1=\alpha_2$ if $\alpha_2\succ\alpha_1$ and $\alpha_2\prec\alpha_1$;
\item  letting $\{\alpha_n\}_{n\in\mathbb{N}}$ be a sequence of measurable partitions of $(X,\mathscr{B}_X,\mu)$,  denote $\wedge_{i\in \mathbb{N}}\alpha_i$ ($\vee_{i\in \mathbb{N}}\alpha_i$) as the most the finer (coarser) measurable partition which  is coarser (finer) than  all $\alpha_i$.
\end{enumerate}
In fact,   sub $\sigma$-algebras of  $\mathscr{B}_X$  are in one-to-one correspondence with the classes of mod-0 equal measurable partitions (see the details in \cite[Chapter 0 \S 2]{MR1369243}). Now, we review   a useful tool  in this paper: measure disintegration. 
\begin{lemma}[{\cite[Proposition 5.19]{EW}}]
\label{23-9-29-1511}
Let $\alpha$  be a measurable partition of Polish probability space $(X,\mathscr{B}_X,\mu)$.  Then,  there is a family of conditional probability measures $\{\mu^{\alpha}_x\}_{x\in X}$ on $(X,\mathscr{B}_X)$ which are characterized by
\begin{itemize}
	\item   there exists a $\mu$-full measure subset $X'$ such that for any $x\in X'$ one has that $\mu^{\alpha}_x(\alpha(x))=1$ and $\mu_{x_1}^{\alpha}=\mu_{x_2}^{\alpha}$, where $\alpha(x)$ is the  unique atom in $\alpha$ which contains $x$, and $x_1, x_2\in \alpha(x)$;
	\item for each $f \in L^1(X,\mathscr{B}_X,\mu)$,  one  has that $f \in L^1(X,\mathscr{B}_X,\mu^{\alpha}_x)$ for $\mu$-a.s. $x\in X$, the map $ x\mapsto
	\int_X f\,\mathrm{d}\mu_x^{\alpha}$ belongs to $L^1(X,\mathscr{B}_X,\mu)$ and $\mu=\int_X\mu_x^{\alpha}\mathrm{d}\mu(x)$ in the sense that
	$$\int_X \left(\int_X f\,\mathrm{d}\mu_x^{\alpha} \right)\, \mathrm{d}\mu(x)=\int_X f \,\mathrm{d}\mu.$$
\end{itemize}
Then  $\mu=\int_X\mu_x^{\alpha}\mathrm{d}\mu(x)$  is called \emph{disintegration} of $\mu$ with respect to $\alpha$.  
\end{lemma}

Next,  we recall some basic definitions of measure-preserving dynamical systems.  For sake of convenience,   denote  $\mathbf{T}$ as $\mathbb{R},\mathbb{R}_+:=\{t\in\mathbb{R}: t\geq0\},\mathbb{Z}$ or $\mathbb{Z}_+:=\{t\in\mathbb{Z}: t\geq0\}$. 
 \begin{definition}
A \emph{measurable dynamical system} $(X,\mathscr{B}_X, (T_t)_{t\in\mathbf{T}})$ is said  that  the mapping 
$$ T: \mathbf{T}\times X\to X, \quad (t,x)\mapsto  T_tx$$
is   $\mathscr{B}_{\mathbf{T}}\otimes\mathscr{B}_X/\mathscr{B}_X$-measurable satisfying that  $T_0=\text{id}_X$ and  $T_t\circ T_s= T_{t+s}$ for any $s,t\in\mathbf{T}$.   If $\mathbf{T}=\mathbb{R}$ or $\mathbb{Z}$,   then $(X,\mathscr{B}_X, T)$ is called an invertible measurable dynamical system.  
\end{definition} 

\begin{definition}
Let  $(X,\mathscr{B}_X, (T_t)_{t\in\mathbf{T}})$ be a measurable dynamical system and  $(X,\mathscr{B}_X,\mu)$ be a Polish probability space.
\begin{enumerate}[(a)]
\item  the measure $\mu$  is called a $(T_t)_{t\in\mathbf{T}}$-\emph{invariant measure} if  $(T_t)_*\mu=\mu$  for every $t\in \mathbf{T}$;
\item  the measure $\mu$  is called a $(T_t)_{t\in\mathbf{T}}$-\emph{ergodic  measure}  if  (1) $\mu$  is a  $(T_t)_{t\in\mathbf{T}}$-invariant measure;  (2) any  measurable subset $A$ with $\mu(A\Delta (T_t)^{-1}A)=0$ for any $t\in\mathbf{T}$  has $\mu$-full measure or $\mu$-null measure.  
\end{enumerate}
The quadruple  $(X,\mathscr{B}_X,\mu, (T_t)_{t\in\mathbf{T}})$ is refer to  a  \emph{measure-preserving dynamical system} or \emph{metric system}. Sometimes, we will  write $(X,\mathscr{B}_X,\mu,  (T_t)_{t\in\mathbf{T}})$ short as $(X,\mu,  (T_t)_{t\in\mathbf{T}})$.  
  \end{definition}

A random dynamical system  (being abbreviated as RDS) $\varphi$ on a Polish space $M$ over a metric system  $(\Omega,\mathscr{F},\mathbb{P},(\theta_t)_{t\in\mathbf{T}})$  means that 
$$\varphi: \mathbf{T}\times\Omega\times M\to M, \quad (n, \omega,x)\mapsto \varphi_n(\omega)x$$
is a  mapping  satisfying  that 
\begin{enumerate}[(i)]
\item $\varphi$ is $\mathscr{B}_{\mathbf{T}}\otimes\mathscr{F}\otimes\mathscr{B}_M/\mathscr{B}_M$-measurable;
\item   for $\mathbb{P}$-a.s. $\omega\in\Omega$,  one has that $\varphi_0(\omega)=\text{id}_M$  and $\varphi_{n+m}(\omega)= \varphi_n(\theta_m\omega)\circ \varphi_m(\omega)$  for any $n,m\in\mathbf{T}$.
\end{enumerate}
The RDS $\varphi$ induces a  measurable  dynamical system $(\Omega\times M,\mathscr{F}\otimes\mathscr{B}_M, (\Phi_t)_{t\in\mathbf{T}})$ given by 
\begin{align}
\label{23-8-15-1253}
\Phi: \mathbf{T}\times\Omega\times M\to \Omega\times M, \quad (t, \omega,x)\mapsto(\theta_t\omega, \varphi_t(\omega)x).
\end{align}
 If  $\mathbf{T}$ is  $\mathbb{R}_+$ or $\mathbb{Z}_+$ ($\mathbb{R}$ or $\mathbb{Z}$), the RDS $\varphi$ is called a \emph{one-side} (\emph{two-side}) RDS. In the case of  $\mathbf{T}=\mathbb{Z}$ or $\mathbb{R}$, note that the induced measurable dynamical system $(\Omega\times M,\mathscr{F}\otimes\mathscr{B}_M, (\Phi_t)_{t\in\mathbf{T}})$ is invertible. 
\begin{definition}
Let $\varphi$ be a measurable RDS on a Polish space $M$ over a metric system $(\Omega,\mathscr{F},\mathbb{P}, (\theta_t)_{t\in\mathbf{T}})$.  A probability measure $\mu$ on $(\Omega\times M,\mathscr{F}\otimes\mathscr{B}_M)$ is an invariant (ergodic) measure of the RDS $\varphi$ if  $\mu$ is   $(\Phi_t)_{t\in\mathbf{T}}$-invariant (ergodic) and   the marginal measure on $\Omega$ is the measure $\mathbb{P}$.  The pair $(\varphi,\mu)$ is refer to a  \emph{measure-preserving} RDS on $M$ over $(\Omega,\mathscr{F},\mathbb{P},(\theta_t)_{t\in\mathbf{T}})$. 
\end{definition}

In the end of this subsection,  we recall  measure disintegration of RDS, which can be regard as a special version of  \Cref{23-9-29-1511}.
\begin{lemma}[{\cite[Lemma 1.4.3]{A}}]
\label{23-8-14-1453}
Let $(\varphi,\mu)$ be a measure-preserving RDS on a Polish space  $M$ over a metric system $(\Omega,\mathscr{F},\mathbb{P}, (\theta_t)_{t\in\mathbf{T}})$.   The disintegration of $\mu$ with respect to measurable partition $\varpi=\{\omega\times M: \omega\in\Omega\}$  can be expressed as 
\begin{align}
\label{23-8-30-2027}
 \mu=\int_{\Omega}\delta_{\omega}\times\mu_{\omega}\mathrm{d}\mathbb{P}(\omega),
\end{align}
where $\delta_{\omega}$ is the Dirac's measure on $(\Omega,\mathscr{F})$ at $\omega$, and $\{\mu_{\omega}\}_{\omega\in \Omega}$ is  a family of probability measures  on $(M,\mathscr{B}_M)$. 
\end{lemma}

\subsection{Entropy and  $K$-system for RDS}

 In the subsection,  we will give the definitions of entropy and  $K$-system for RDS.  Assume that $(\varphi,\mu)$ is a  measure-preserving RDS on a Polish space $M$ over a metric dynamical system $(\Omega,\mathscr{F},\mathbb{P},(\theta_t)_{t\in\mathbf{T}})$ through this subsection,  where $\mathbb{T}=\mathbb{Z}$ or $\mathbb{Z}_+$.  Firstly, recall that 
  \begin{itemize}
  \item $\varpi=\{\omega\times M: \omega\in\Omega\}$ is the measurable  partition of $(\Omega\times M, \mathscr{F}\otimes\mathscr{B}_M, )$;
  \item $\mu=\int_{\Omega}\delta_{\omega}\times\mu_{\omega}\mathrm{d}\mathbb{P}(\omega)$ is the disintegration of $\mu$ with respect to $\varpi$ which is defined as \eqref{23-8-30-2027}; 
  \item  $(\Phi_{t})_{t\in\mathbf{T}}$ is measurable map on $(\Omega\times M, \mathscr{F}\otimes\mathscr{B}_M)$ which is defined as \eqref{23-8-15-1253}.
  \end{itemize}
\begin{definition}
For any finite measurable partition $\alpha$ of $(\Omega\times M,\mathscr{F}\otimes\mathscr{B}_M)$,  its entropy with respect to the RDS $(\varphi,\mu)$ is defined as 
$$h_{\mu}(\varphi, \alpha):=\int_{\Omega}\lim_{n\to+\infty}\frac{1}{n}H_{\delta_{\omega}\times\mu_{\omega}}(\bigvee_{i=0}^{n-1}(\Phi_{i})^{-1}\alpha)\mathrm{d}\mathbb{P}(\omega),$$
where $H_{\nu}(\alpha):=-\sum_{A\in\alpha}\nu(A)\log\nu(A)$ with  a probability measure $\nu$ on $(\Omega\times M,\mathscr{F}\otimes\mathscr{B}_M)$.  The measure-theoretic entropy of the measure-preserving RDS  $(\varphi,\mu)$ is defined as 
\begin{align}
\label{23-2-12-22-16} 
h_{\mu}(\varphi):=\sup_{\alpha}h_{\mu}(\varphi, \alpha),
\end{align}
where  $\alpha$ is taken all over finite Borel measurable partitions of $\Omega\times M$.
\end{definition}

\begin{rem}
For a measure-preserving dynamical system $(X,\mu, (T_t)_{t\in\mathbf{T}})$ on the Polish probability space, its  $(T_t)_{t\in\mathbf{T}}$-invariant $\sigma$-algebra is one to one the   $(T_t)_{t\in\mathbf{T}}$-invariant measurable partition (for example,  see \cite[Theorem 6.5]{EW}).  Therefore,  the definition of entropy in  \eqref{23-2-12-22-16} is coincide with the classical definition of  relative entropy (see the detail discussions in \cite[Chapter 0]{MR1369243}).  
\end{rem}

 The   Pinsker $\sigma$-algebra $\mathscr{P}_\mu(\varphi)$    of measure-preserving RDS $(\varphi,\mu)$  is defined as the smallest $\sigma$-algebra containing
 $$\big\{ A\in\mathscr{F}\otimes\mathscr{B}_M:  h_\mu(\varphi,\{A,  \Omega\times M\setminus A\})=0\big\},$$
 which  is a $(\Phi_t)_{t\in\mathbf{T}}$-invariant sub-$\sigma$-algebra of $\mathscr{F}\otimes\mathscr{B}_M$ (for example, see \cite[Section 4.10]{Wal} or \cite{MR1786718}). Denote $\mathcal{P}_{\mu}$   as a  $(\Phi_t)_{t\in\mathbf{T}}$-invariant unique measurable partition   determined  by the  Pinsker $\sigma$-algebra $\mathscr{P}_\mu(\varphi)$, which is called \emph{Pinsker partition} of the measure-preserving RDS $(\varphi,\mu)$.  
 
  \begin{definition}[$K$-system]
A measure-preserving RDS  $(\varphi,\mu)$ is called a  $K$-system  if the  Pinsker partition $\mathcal{P}_{\mu}$  of  $(\varphi, \mu)$  is trivial. That is  $\mathcal{P}_{\mu}=\varpi$. 
  \end{definition}
 
 Finally,  we review a proposition which states a relationship between  Pinsker partition and general partition for RDS.  The reader can see a general form in \cite[Lemma 3.6]{MR2197376}. 
   
 \begin{proposition}
 \label{23-6-29-1659}
Assume that $(\varphi,\mu)$ is a two-side measure-preserving RDS on a Polish space over a metric system $(\Omega,\mathscr{F},\mathbb{P},(\theta_t)_{t\in\mathbf{T}})$. Then for any measurable partition $\eta$  of  $(\Omega\times M, \mathscr{F}\otimes\mathscr{B}_M)$ with 
$$\varpi\prec \eta\prec \Phi_1\eta, \quad \bigvee_{i\in\mathbb{N}}\Phi_{i}\eta=\varepsilon_{\Omega\times M}$$
implies that $\bigwedge_{i\in\mathbb{N}}\Phi_{-i}\eta\succ \mathcal{P}_{\mu}$, where $\varepsilon_{\Omega\times M}$ is the extreme measurable partition, i.e. $\varepsilon_{\Omega\times M}=\{(\omega,x)\in \Omega\times M\}$.

  \end{proposition}

\section{Hyperbolic RDS is   a K-system}
\label{23-8-13-2002}
In this section, we limit our setting on abstract smooth RDSs and  we prove that a $C^2$ RDS with a   totally ergodic  non-uniformly hyperbolic  SRB measure  is a $K$-system, namely \Cref{23-6-27-1553}.  The main idea of this proof  is inspired by  the Pesin's spectral decomposition theorem for determinant systems (for example, see \cite{MR0466791, MR3076414}).  Throughout this section,  we assume that 
\begin{enumerate}
\item[(A1)] \label{23-9-29-1456-1} $M$ is a $d$-dimensional  closed connected  Riemannian manifold with $d\geq2$, and $\varphi$ is a RDS on $M$ over a metric system $(\Omega,\mathscr{F},\mathbb{P},(\theta_t)_{t\in\mathbf{T}})$ satisfying that   $\varphi_1(\omega)$ is a $C^2$ diffeomorphism on $M$  for $\mathbb{P}$-a.s. $\omega\in\Omega$  and $\mathbf{T}=\mathbb{Z}$ or $\mathbb{Z}_+$;   
\item[(A2)] \label{23-9-29-1456-2} the RDS $\varphi$  satisfies  following integrability condition
\begin{align}
\label{23-7-5-2042}
\int_{\Omega}\big(\log^+\|\varphi_1(\omega)\|_{C^2}+\log^+\|\varphi_1({\omega})^{-1}\|_{C^2}\big)\mathrm{d}\mathbb{P}(\omega)<+\infty,
\end{align}
where $\log ^+|a|=\max\{\log|a|,0\}$ with $a\in\mathbb{R}$;
\item[(A3)] \label{23-9-29-1456-3}the Borel probability measure $\mu$  on $\Omega\times M$ is an ergodic measure of RDS $\varphi$.
\end{enumerate}

\subsection{Lyapunov exponets and entropy}
In this section, we mainly review the  relationship  between Lyapunov exponents and entropy for RDSs.  Firstly, we recall the well-known multiplicative ergodic theorem, which ensures the existence   of  Lyapunov exponents. The reader  can refer to \cite[Theorem 3.4.1, Theorem 4.2.6 and Theorem 4.3.14]{A} for the details.

\begin{proposition}
\label{23-8-13-1628}
Let $(\varphi,\mu)$ be a measure-preserving  RDS on $M$ over a metric system $(\Omega,\mathscr{F},\mathbb{P},(\theta_t)_{t\in\mathbf{T}})$  satisfying (\hyperref[23-9-29-1456-1]{A1})-(\hyperref[23-9-29-1456-1]{A3}). 

\noindent\textbf{One-side:} assume that $\varphi$ is one-side RDS.  Then there exist $r\in\{1,\dots,d\}$,  constant values $\lambda_1>\cdots>\lambda_{r}$, $m_1,\dots, m_r\in\mathbb{N}$ and  a measurable filtration 
$$TM=:V_1(\omega,x)\supsetneq V_2(\omega,x)\supsetneq\cdots \supsetneq V_{r}(\omega,x)\supsetneq V_{r+1}(\omega,x):=\{0\}$$
with following properties  for $\mathbb{\mu}$-a.s. $(\omega,x)\in\Omega\times M$ and any  $i\in\{1,\dots, r\}$:
\begin{enumerate}[(a)]
\item   for  $v\in V_{i}(\omega,x)\setminus V_{i+1}(\omega,x)$, one has that 
$$\lim_{n\to+\infty}\frac{1}{n}\log|d_x\varphi_n(\omega)v|=\lambda_i;$$
\item $d_x\varphi_n(\omega)V_i(\omega,x)=V_i(\Phi_n(\omega,x))$ and   $m_i=\dim V_{i}(\omega,x)-\dim V_{i+1}(\omega,x)$, where $\Phi_n$ is defined as \eqref{23-8-15-1253}.
\end{enumerate}
\textbf{Two-side:} assume that $\varphi$  is  two-side RDS.  Then there exist $r\in\{1,\dots,d\}$,  constant values $\lambda_1>\cdots>\lambda_{r}$, $m_1,\dots, m_r\in\mathbb{N}$ and  a measurable filtration 
$$TM=E_1(\omega,x)\oplus E_2(\omega,x)\oplus\cdots\oplus E_{r}(\omega,x)$$
with following properties  for $\mathbb{\mu}$-a.s. $(\omega,x)\in\Omega\times M$ and any  $i\in\{1,\dots,r\}$:
\begin{enumerate}[(a)]
\item for  all $v\in E_{i}(\omega,x)$, one has that 
$$\lim_{n\to\pm\infty}\frac{1}{n}\log|d_x\varphi_n(\omega)v|=\lambda_i;$$
\item $d_x\varphi_n(\omega)E_i(\omega,x)=E_i(\Phi_n(\omega,x))$ and  $m_i=\dim E_{i}(\omega,x)$;
\item  denoting $\pi_i(\omega, x)$ as the projection on $E_i(\omega, x)$ along $\bigoplus_{j\in\{1,\dots,r\}\setminus i}E_j(\omega,x)$,  then  
$$\lim_{n\to\pm\infty}\frac{1}{n}\log|\pi_{i}(\Phi_n(\omega,x))|=0.$$
\end{enumerate}
The constants  $\lambda_i$  are called Lyapunov exponents of $(\varphi,\mu)$ with multiplicities $m_i$. 
\end{proposition}

Let  $(\varphi,\mu)$  be a  one-side  measure-preserving RDS on $M$ over  $(\Omega,\mathscr{F}, \mathbb{P}, (\theta_n)_{n\in\mathbb{Z}_+})$ which satisfies (\hyperref[23-9-29-1456-1]{A1})-(\hyperref[23-9-29-1456-1]{A3}).  In the following we will give a natural way to define   a two-side measure-preserving  RDS.  By \cite{MR0143873, MR0217258},  there exists an ergodic measure-preserving dynamical system  $(\hat{\Omega},\hat{\mathscr{F}}, \hat{\mathbb{P}}, (\hat{\theta}_n)_{n\in\mathbb{Z}})$ on the Polish probability space $(\hat{\Omega},\hat{\mathscr{F}}, \hat{\mathbb{P}})$  and  a factor map  
\begin{align}
\label{23-10-17-0002}
\hat{\pi}:  (\hat{\Omega}, \hat{\mathscr{F}}, \hat{\mathbb{P}}, (\hat{\theta}_n)_{n\in\mathbb{Z}})\to (\Omega,\mathscr{F}, \mathbb{P}, (\theta_n)_{n\in\mathbb{Z}_+}).
\end{align}
Usually, \eqref{23-10-17-0002} is called  the invertible  extension of $(\Omega,\mathscr{F}, \mathbb{P},(\theta_n)_{n\in\mathbb{Z}_+})$.  Then define a two-side  RDS $\hat{\varphi}$  on $M$ over $ (\hat{\Omega},\hat{\mathscr{F}}, \hat{\mathbb{P}},(\hat{\theta}_n)_{n\in\mathbb{Z}})$ as 
\begin{align}
\label{23-8-15-1413}
\hat{\varphi}_n(\hat{\omega})x=
\begin{cases}
\varphi_n\big(\hat{\pi}(\hat{\omega})\big)x \quad&\text{ if } n\geq0
\\ \varphi_{-n}\big(\hat{\pi}(\hat{\theta}_{n}\hat{\omega})\big)^{-1}x\quad&\text{ if } n<0.
\end{cases}
\end{align}
By \cite[Theorem 1.7.2]{A}, it is clear that there exists a unique ergodic measure $\hat{\mu}$   of RDS $\hat{\varphi}$ with the marginal $\mu$ on $\Omega\times M$. In this paper, we call  two-side measure-preserving RDS $(\hat{\varphi},\hat{\mu})$ as invertible extension of the measure-preserving  RDS $(\varphi,\mu)$.  The reader should be caution that  the main reason why we can do as this is that  $\varphi_1(\omega)$ is a diffeomorphism on $M$ for $\mathbb{P}$-a.s. $\omega\in\Omega$.  According to  \Cref{23-8-13-1628} and  a similar argument in \cite[Theorem 2.3 in Chapter I]{MR1369243}, we have the following remark. 
\begin{rem}
\label{23-8-13-1921}
The entropy and the Lyapunov exponents of $(\varphi,\mu)$  are coincides with the Lyapunov exponents and entropy of $(\hat{\varphi},\hat{\mu})$, respectively. 
\end{rem}

Finally,  we review the celebrate Pesin's entropy formula and its characterization in  RDSs.
\begin{proposition}
\label{23-8-15-1444}
Assuming that $(\varphi,\mu)$ is a one-side RDS satisfying (\hyperref[23-9-29-1456-1]{A1})-(\hyperref[23-9-29-1456-1]{A3}),  then $\mu$ is a SRB measure (see \Cref{23-8-13-1643}) if and only if following equality holds, 
$$h_{\mu}(\varphi)=\sum_{i=1}^r\lambda_i^+m_i,$$
where $a^+:=\max\{a,0\}$ with $a\in\mathbb{R}$.  Particularly,  if $\mu=\mathbb{P}\times\text{vol}$ is an ergodic measure of RDS $\varphi$, where $\text{vol}$ is the volume measure on $M$, then $\mu$ and $\hat{\mu}$  are both  SRB measure of RDS $\varphi$  and $\hat{\varphi}$, respectively, where $(\hat{\varphi},\hat{\mu})$ is the invertible extension of the RDS $(\varphi,\mu)$.
\end{proposition}
\begin{proof}
Note that 
\begin{align*}
-\log^+|\varphi_{1}(\omega)^{-1}|_{C^2}\lesssim \log\inf_{x\in M}|\det d_x\varphi_1(\omega)|\lesssim\log^+|\varphi_1(\omega)|_{C^2}.
\end{align*}
By \cite[Theorem 2.2 and ]{MR1944407} or \cite[Theorem 3.13]{KiL}, one has that  $\mu$ is a SRB measure if and only if  $h_{\mu}(\varphi)=\sum_{i=1}^r\lambda_i^+m_i$.  If  $\mu=\mathbb{P}\times\text{vol}$,  then $\mu$ and $\hat{\mu}$  are both  random SRB measures, since \Cref{23-8-13-1921} and \cite[Corollary 2.3]{MR1944407}.  This completes the proof of the Proposition.
\end{proof}

\subsection{$K$-system}
Recall that $(\varphi,\mu)$ is the  measure-preserving RDS on $M$ over the metric  system $(\Omega,\mathscr{F},\mathbb{P},(\theta_t)_{t\in\mathbf{T}})$ and it satisfies (\hyperref[23-9-29-1456-1]{A1})-(\hyperref[23-9-29-1456-3]{A3}).  In this subsection,  we additionally assume that 
\begin{enumerate}
\item[(A4)]\label{23-9-29-1456-4} $\varphi$  is a two-side measure-preserving RDS;
\item[(A5)]\label{23-9-29-1456-5}  $\mu$ is an ergodic hyperbolic measure of $\varphi$, i.e. $(\varphi,\mu)$ has non-zero Lyapunov exponents;
\item[(A6)]\label{23-9-29-1456-6} $(\varphi,\mu)$ is totally ergodic, i.e.  its induced measure-preserving dynamical system $(\Omega\times M, (\Phi_{nm})_{m\in\mathbb{Z}},\mu)$ is ergodic for any $n\in\mathbb{N}$.
 \end{enumerate}
In the next, we will use the invariant manifold theorem to prove that $(\varphi,\mu)$ is a  $K$-system.  Invariant manifold theorem describes  geometry  structures of stables sets  and unstable sets defined as following, 
\begin{align*}
&W^s_{(\omega,x)}=\{y\in M:  \limsup_{n\to+\infty}\frac{1}{n}\log \text{dist}(\varphi_{n}(\omega)y, \varphi_{n}(\omega)x)<0\}
\\&W^u_{(\omega,x)}=\{y\in M:  \limsup_{n\to+\infty}\frac{1}{n}\log \text{dist}(\varphi_{-n}(\omega)y, \varphi_{-n}(\omega)x)<0\},
\end{align*}
where $\text{dist}$ is the metric induced by the Riemannian structure of $M$.  Due to multiplicative  ergodic theorem (see \Cref{23-8-13-1628}),  denote 
 $$ E^s_{(\omega,x)}:=\bigoplus_{\lambda_i<0}E_{i}{(\omega,x)}\quad\text{and}\quad  E^u_{(\omega,x)}:=\bigoplus_{\lambda_i>0}E_{i}{(\omega,x)} $$
 as the stable space and unstable space at $(\omega, x)$ of $(\varphi,\mu)$, respectively.
Now, we state a convenient version of  invariant manifold theorem. The reader can refer  to \cite{MR0805125} or \cite[Chapter V]{MR1369243} for an analogous  proof.
\begin{proposition}[Invariant manifold theorem]
\label{23-09-29-1641}
Assume that the measure-preserving RDS $(\varphi,\mu)$ satisfies that (\hyperref[23-9-29-1456-1]{A1})-(\hyperref[23-9-29-1456-5]{A5}).  Given a $\lambda\in(0, \min_{i=1,\dots,r}|\lambda_i|)$, there exist two tempered functions $l, C: \Omega\times M\to [1,+\infty)$\footnote{A tempered function $l$ on $\Omega\times M$ is that $l$  is a Borel measurable function and $\lim_{n\to\pm\infty}\frac{\log|l(\Phi_n(\omega,x))|}{n}=0.$} such that for any enough small constant $\delta>0$ there is a  $\Phi$-invariant $\mu$-full measure  set $\Gamma$ and a  unique family of   $C^{1+Lip}$ maps 
$$\{g^{\tau}_{(\omega, x)}: E^{\tau}_{(\omega, x)}(\delta l(\omega,x)^{-1})\to E^{\tau'}_{(\omega, x)}\}_{(\omega,x)\in\Gamma},$$ where  $\tau\neq\tau'\in\{u, s\}$, with following properties 
\begin{enumerate}[(a)]
\item $d_0g^{\tau}_{(\omega, x)}=0$;
\item $\text{Lip}(g^{\tau}_{(\omega, x)})\leq 1/10$,  and $\text{Lip}(d_0g^{\tau}_{(\omega, x)})\leq C(\omega, x)l(\omega, x)$;
\item  for any  $n\in\mathbb{N}$ one has  that 
\begin{align*}
&\text{dist}(\varphi_{-n}(\omega)y_1,  \varphi_{-n}(\omega)y_2)\leq C(\omega,x)e^{-n\lambda}d(y_1,y_2) \quad \text{if }y_1,y_2\in W^{u}_{(\omega,x),\delta}:=\exp_x(\text{graph }g^{u}_{(\omega, x)}),
\\& \text{dist}(\varphi_n(\omega)y_1,  \varphi_n(\omega)y_2)\leq C(\omega, x)e^{-n\lambda}d(y_1,y_2)\qquad \text{if } y_1,y_2\in W^{s}_{(\omega,x),\delta}:=\exp_x(\text{graph }g^{s}_{(\omega, x)}),
\end{align*}
where $\exp_x$ is the exponential map at $x$; 
\item  one has that 
\begin{align}
W^u_{(\omega,x)}=\bigcup_{n\in\mathbb{Z}_+}\varphi_n(\theta_{-n}\omega)W^u_{(\Phi_{-n}(\omega,x)),\delta}, \quad W^s_{(\omega,x)}=\bigcup_{n\in\mathbb{Z}_+}\varphi_{-n}(\theta_{n}\omega)W^s_{(\Phi_{n}(\omega,x)),\delta},
\end{align}
which are immersed submanifolds of $M$.
\end{enumerate}
\end{proposition}
In the above  proposition, we usually call $W^{u}_{(\omega,x),\delta}$ ($W^{s}_{(\omega,x),\delta}$) as  local unstable (stable) manifolds at $(\omega, x)$ of $(\varphi,\mu)$.   Next, we mainly review the construction process of  measurable partitions which  are  subordinated  to the unstable or stable manifolds.  In the  following, we always assume that  $(\varphi,\mu)$ is a measure-preserving RDS  on $M$ over $(\Omega,\mathscr{F},\mathbb{P},(\theta_t)_{t\in\mathbf{T}})$, and it satisfies that (\hyperref[23-9-29-1456-1]{A1})-(\hyperref[23-9-29-1456-5]{A5}) without extra specific explanation.
\begin{definition}
A measurable partition $\eta$ of $\Omega\times M$  is subordinated to the unstable manifolds of $(\varphi,\mu)$  if there exists a positive measurable function $r: \Omega\times M\to \mathbb{R}$ such that for $\mu$-a.s. $(\omega,x)\in\Omega\times M$,  one has that 
$$\mathcal{B}^{u}_{(\omega,x)}(r(\omega,x))\subset \eta^{\omega}(x):=\{y\in M: (\omega,y)\in\eta(\omega, x) \}\subset W^{u}_{(\omega,x)},$$
where $\mathcal{B}^{u}_{(\omega,x)}(r(\omega,x))$ is the open ball at $(\omega,x)$ with the radius $r(\omega,x)$ on the immersed manifold $W^{u}_{(\omega,x)}$.  The definition of measurable partition which is subordinated to the stable manifolds is similar.
\end{definition}

 Using  Pesin's block, Lusin's theorem  and  properties of invariant manifolds,  we can obtain  following  claims.  Since the argument is standard in smooth ergodic theory,  the    proof  will be omitted.  The reader can refer to \cite[Chapter IV]{MR1369243} and \cite{MR3925383}  for the  details.

 \begin{claim}[Local properties] 
 \label{23-8-11-2111}
Recall that $\Gamma$ is a $\mu$-full measure subset of $\Omega\times M$ defined as \Cref{23-09-29-1641}. Then there exists a positive  $\mu$-measure compact subset $\Gamma_1\subset\Gamma$ such that for any   $\tau\neq\tau'\in\{u,s\}$,  for all $r>0$, there is a sufficiently small positive constant $\epsilon=\epsilon(r)$  and  a point $(\check{\omega}, \check{x})\in \Gamma_1$ such that 
$$\mathcal{L}(r):=\Gamma_1\cap \big(\overline{\mathcal{B}(\check{\omega}, \epsilon)}\times\overline{\mathcal{B}(\check{x}, \epsilon)}\big)$$ with following properties 
\begin{enumerate}[(1)]
\item $\mu(\mathcal{L}(r))>0$; 
\item  for any $(\omega,x)\in\mathcal{L}(r)$,  there is a $C^{1+Lip}$ map $g^{\tau}_{(\omega,x)}: E_{(\check{\omega},\check{x})}^{\tau}(r)\to  E_{(\check{\omega},\check{x})}^{\tau'}(r)$ with $\text{Lip}(g_{(\omega,x)})\leq 1$ such that the connected component of $W^{\tau}_{ (\omega,x),\delta}\cap \overline{B(\check{x}, \epsilon)}$ containing $x$ coincides with $\exp_{\check{x}}(\text{graph }g^{\tau}_{(\omega,x)})$;
\item  $(\omega,x)\mapsto g^{\tau}_{(\omega,x)}$ varies continuously in the uniform norm on $C(E^{\tau}_{(\check{\omega},\check{x})}(r), E^{\tau'}_{(\check{\omega},\check{x})}(r))$ as $(\omega, x)$ varies in $\mathcal{L}(r)$, where $C(X, Y)$ is  collections of  the continuous maps from $X$ to $Y$ with compact metric spaces $X,Y$;
\item  for any $(\omega,x)\in\mathcal{L}(r)$,  submanifold $\text{graph } g_{(\omega,x)}^{\tau}$ is transverse to the family of manifolds $\{\text{graph } g_{(\omega,y)}^{\tau'}\}_{(\omega, y)\in \mathcal{L}(r)}$.
\end{enumerate}
\end{claim}
 \begin{claim}[Subordinated partitions] 
 \label{23-7-5-1453}
For any $\tau\in\{u,s\}$,  define   $\tau$-random stack  as 
\begin{align}
\label{23-7-5-1502}
\mathcal{S}_{\tau}(r):=\bigcup_{(\omega,x)\in\mathcal{L}(r)}\{\omega\}\times \exp_{\check{x}}(\text{graph }g^{\tau}_{(\omega,x)}) 
\end{align}
which is a positive $\mu$-measure subset of  $\Omega\times M$, and  measurable partition as
 $$\xi_{\tau}(r):=\{\omega\times \exp_{\check{x}}(\text{graph }g^{\tau}_{(\omega,x)}): (\omega,x)\in \mathcal{L}(r)\}\cup\{\Omega\times M\setminus \mathcal{S}_{\tau}(r)\} .$$
Then there exists $r_1>0$ such that   $\eta_s:=\bigvee_{i=0}^{+\infty}\Phi_{-i}\xi_{s}(r_1)\vee\varpi$  and $\eta_u:=\bigvee_{i=0}^{+\infty}\Phi_i\xi_{u}(r_1)\vee\varpi$  are two  measurable partitions of $(\Omega\times M,\mathscr{F}\otimes\mathscr{B}_M,\mu)$, where $\varpi=\{\omega\times M: \omega\in \Omega\}$,  which  satisfy  that 
\begin{enumerate}[(a)]
\item  $\varpi\prec \eta_s\prec\Phi_{1}\eta_s$ and $\bigvee_{n\in\mathbb{N}}\Phi_{n}\eta_s=\varepsilon_{\Omega\times M}$;
\item  $\varpi\prec \eta_u\prec\Phi_{-1}\eta_u$ and $\bigvee_{n\in\mathbb{N}}\Phi_{-n}\eta_u=\varepsilon_{\Omega\times M}$;
\item  $\eta_u$ and $\eta_s$ are subordinated to the unstable  and stable manifolds, respectively.
\end{enumerate}
\end{claim}


Finally, we are going to state the main result (\Cref{23-6-27-1553}) in this section.    Despite  there is a similar result in     \cite[Theorem C]{MR968818} for i.i.d. RDS as \Cref{23-6-27-1553},  we are dealing  with general RDSs.   For  sake of convenience,   we introduce two simplified notations:  
\begin{itemize}
 \item  given  a measurable partition $\eta$ of $(\Omega\times M,\mathscr{F}\otimes\mathscr{B}_M,\mu)$ with $\eta\succ\varpi$,   for any  $\omega\in\Omega$  denote 
$\eta^{\omega}$ as  projection measurable partition on $M$, i.e. 
\begin{align}
\eta^{\omega}=\{A\in\mathscr{B}_M: \{\omega\}\times A\in\eta\};
\end{align}
\item given a Borel measurable subset $B$ of $\Omega\times M$,  for any  $\omega\in\Omega$  denote $B^{\omega}$ as  projection set, i.e. 
\begin{align}
B^{\omega}=\{x\in M: (\omega, x)\in B\}.
\end{align}
\end{itemize}
\begin{definition}
\label{23-8-13-1643}
An ergodic  measure $\mu$ of RDS $\varphi$ is called a  SRB measure if  for any measurable partition $\eta$ which is subordinated to the stable manifolds of $(\varphi,\mu)$,  one has that 
$$(\mu_{\omega})^{\eta^{\omega}}_x\ll \text{vol}_{(\omega,x)}^{s}$$
for $\mu$-a.s. $(\omega,x)\in\Omega\times M$, where $\{(\mu_{\omega})^{\eta^{\omega}}_x\}_{x\in M}$  is the disintegration of $\mu_{\omega}$ with respect to $\eta^{\omega}$, and $\text{vol}_{(\omega,x)}^{s}$ is the volume measure on $W^{s}_{(\omega,x)}$ induced by its inherited Riemannian metric as a submanifold of $M$.
\end{definition}

\begin{proposition}
\label{23-6-27-1553}
Assume that $(\varphi,\mu)$ is a  measure-preserving RDS on $M$ over a metric system $(\Omega,\mathscr{F},\mathbb{P},(\theta_n)_{n\in\mathbb{Z}})$ satisfying  (\hyperref[23-9-29-1456-1]{A1})-(\hyperref[23-9-29-1456-6]{A6}). If $\mu$ is a  SRB measure,  then  $(\varphi,\mu)$ is a  $K$-system.
\end{proposition}
\begin{proof}
Let  $\eta_s$  and $\eta_u$ be two measurable partitions defined as  \Cref{23-7-5-1453}.  Due to the properties  in \Cref{23-8-11-2111} and SRB property of  $\mu$, it is easy to see that there exists a positive $\mu$-measure compact subset $\mathcal{L}_1$ of $\mathcal{L}(r_1)$  and $r_2>0$ such that for every $(\omega,x)\in\mathcal{L}_1$
\begin{itemize}
\item [(H1)]\label{23-9-2-1530-2} for any $\tau\in\{s,u\}$ one has that  $\mathcal{B}^{\tau}_{(\omega,x)}(r_2)\subset \eta_{\tau}(\omega,x)$,  where $\mathcal{B}^{\tau}_{(\omega,x)}(r_2)$ is the open ball at $(\omega,x)$ with the radius $r_2$ on the immersed manifold $W^{\tau}_{(\omega,x)}$;
\item[(H2)]\label{23-9-2-1530-3} $(\mu_{\omega})^{\eta_s^{\omega}}_x(\mathcal{B}^{s}_{(\omega,x)}(r_2))>0$.
\end{itemize}
 
According to the absolutely continuous theorem (see \cite[Chapter III  \S 5]{MR1369243}) and (\hyperref[23-9-2-1530-2]{H1})-(\hyperref[23-9-2-1530-3]{H2}),  there exists a positive  $\mu$-measure subset  $\mathcal{L}_2$ of $\mathcal{L}_1$   such that for any $(\omega,x)\in\mathcal{L}_2$ there is a positive $\mu_{\omega}$-measure    subset $\mathcal{B}_{(\omega,x)}$ of  $\bigcup_{y\in \eta^{\omega}_s( x)} \eta^{\omega}_{u}(y)$.   Therefore,   for every $(\omega,x)\in\mathcal{L}_2$ one has that  $\mu_{\omega}\big((\eta_u\wedge\eta_s)^{\omega}(x)\big)>0$ by using the fact 
\begin{align*}
\eta_s\wedge\eta_u(\omega, x)\supset\bigcup_{y\in \eta^{\omega}_s( x)}\{\omega\}\times \eta^{\omega}_{u}(y).
\end{align*}

Recall that $\mathcal{P}_{\mu}$ is the  Pinsker partition of measure-preserving  RDS $(\varphi,\mu)$.  By  \Cref{23-6-29-1659}, one has that  
  $$\mathcal{P}_{\mu}\prec\bigwedge_{n\in\mathbb{N}}\Phi_{-n}\eta_{s}\text{ and } \mathcal{P}_{\mu}\prec\bigwedge_{n\in\mathbb{N}}\Phi_{n}\eta_{u},$$
   which implies that $\mathcal{P}_{\mu}\prec\eta_s\wedge\eta_u$,  Therefore,  for any  $(\omega, x)\in \mathcal{L}_2$ one has that
  $$\mu_{\omega}\big(\mathcal{P}_{\mu}^{\omega}( x)\big)>0.$$
    By the invariance of $\mathcal{P}_{\mu}$ and  ergodicity of the measure $\mu$,  we have that $\mu_{\omega}\big(\mathcal{P}_{\mu}^{\omega}( x)\big)>0$ for $\mu$-a.s. $(\omega,x)\in\Omega\times M$.  Therefore,    the projective partition $\mathcal{P}_{\mu}^{\omega}$ has at most countable positive $\mu_{\omega}$-measure element  element for $\mathbb{P}$-a.s. $\omega\in\Omega$.  According to \cite[Corollary 4.30]{MR0262464}, there exists a finite Borel measurable partition $\gamma=\{A_1, \dots, A_m\}$ of $\Omega\times M$ such that
    $$\Phi_{j-1}(A_1)=A_j \text{ for any }j=1,\dots,m \quad\text{and}\quad \gamma\vee\varpi=\mathcal{P}_{\mu}.$$  Since  $(\varphi,\mu)$ is totally ergodic and $\Phi_m(A_1)=A_1$, we have that $A_1=\cdots= A_m=\Omega\times M\pmod\mu$. Therefore,  $(\varphi,\mu)$ is a $K$-system. 

\end{proof}
\section{Proof of  \Cref{23-7-4-2042}}
\label{23-8-22-1709}
In this section, we will mainly verify the invertible extension of Lagrangian flow  is a  $K$-system by  borrowing \Cref{23-6-27-1553}.  Combining this property and \Cref{23-7-8-2238},  we  prove   that the Lagrangian flow has observable full-horseshoes.
\subsection{Totally ergodic property}
In this subsection, we mainly  explain that the time-1 map of  Lagrangian flow  induces a totally ergodic discrete RDS.   Recall that  $(\Omega,\mathscr{F}, \mathbb{P})$  is the infinite-dimensional  Wiener space which is defined  as  \eqref{23-8-28-1017}. Now, we define  Wiener shift $(\theta_t)_{t\in\mathbb{R}_+}$ on it as 
\begin{align}
\label{23-8-28-1025}
\theta_t(\omega)=(\omega_k(t+\cdot)-\omega_k(t))_{k\in\mathbb{Z}_0^2},
\end{align}
where $\mathbb{Z}_0^2=\{n\in\mathbb{Z}^2: n\neq(0,0)\}$.  Following the classical argument  of  Bernoulli system being ergodic (for example, see \cite[Proposition 2.15]{EW}),  the  reader can  prove that the  Wiener measure $\mathbb{P}$  is $(\theta_{\tau t})_{t\in\mathbb{R}_+}$-ergodic  for any   $\tau>0$.   In the following,  we  use  $(u_t)_{t\in\mathbb{R}_+}$  and $(\varphi_t)_{t\in\mathbb{R}_+}$  to represent the  velocity flow   generated by \Cref{23-8-4-1509} and  the  Lagrangian  flow 
generated by \Cref{23-7-4-1717}, respectively.  
\begin{lemma}
 Recall that $\mu$ is the unique stationary measure of velocity flow defined in \Cref{23-9-2-2037}.  Then $u_t: \Omega\times \mathbb{H}\to \mathbb{H}$ is a continuous   RDS on $\mathbb{H}$ over the metric system $(\Omega,\mathscr{F},\mathbb{P}, (\theta_t)_{t\in\mathbb{R}_+})$ and $\mathbb{P}\times\mu$ is an invariant measure of RDS $(u_t)_{t\in\mathbb{R}_+}$.  Furthermore,  $\varphi_t$ is a $C^2$ RDS  on $\mathbb{T}^2$ over the metric system $(\Omega\times\mathbb{H},\mathscr{F}\otimes\mathscr{B}_{\mathbb{H}},\mathbb{P}\times\mu, (\theta_t\times u_t)_{t\in\mathbb{R}_+})$.    
 \end{lemma}
 \begin{proof}
 According to  the argument in \cite[Chapter 2.4.4]{MR3443633}, we know that  $(u_t)_{t\geq0}$ is a continuous  RDS on $\mathbb{H}$ over the metric system $(\Omega,\mathscr{F},\mathbb{P}, (\theta_t)_{t\in\mathbb{R}_+})$.   By that and \cite[Theorem 2.1.7]{A}, one has that  $\mathbb{P}\times\mu$ is an invariant measure of RDS $(u_t)_{t\in\mathbb{R}_+}$.  Since $\varphi_t$ is the solution of   a random ordinary differential equation of \eqref{23-8-4-1509} , $(\varphi_t)_{t\in\mathbb{R}_+}$ is a $C^2$ RDS on $\mathbb{T}^2$ over the metric system $(\Omega\times\mathbb{H},\mathscr{F}\otimes\mathscr{B}_{\mathbb{H}},\mathbb{P}\times\mu, (\theta_t\times u_t)_{t\in\mathbb{R}_+})$ by \cite[Theorem 2.2.2]{A}.
 \end{proof}
 Denote  $(\mathscr{U}_t)_{t\in\mathbb{R}_+}$ and $(\Phi_t)_{t\in\mathbb{R}_+}$ as the measurable maps generated by $(u_t)_{t\in\mathbb{R}_+}$ and $(u_t,\varphi_t)_{t\in\mathbb{R}_+}$  as  \eqref{23-8-15-1253}.  Letting  $\text{vol}$ be the volume measure on $\mathbb{T}^2$,  then we have following lemma.
\begin{lemma}
\label{23-8-15-1431}
 For any $m\in\mathbb{N}$,     $(\Omega\times\mathbb{H}\times\mathbb{T}^2,  \mathbb{P}\times\mu\times\text{vol}, (\mathscr{U}_{nm})_{n\in\mathbb{Z}_+})$    is an ergodic measure-preserving dynamical system.  Particularly,  the time-1 map of RDS  $(\varphi, \mathbb{P}\times\mu\times\text{vol}) $  generates a totally ergodic measure-preserving  RDS $(\varphi^{(1)}, \mathbb{P}\times\mu\times\text{vol})$ on  $\mathbb{T}^2$ over the metric system  $(\Omega\times \mathbb{H}, \mathbb{P}\times\mu,(\Phi_{n})_{n\in\mathbb{Z}_+})$ as 
 $$\varphi^{(1)}: \mathbb{Z}_+\times\Omega\times \mathbb{H}\times\mathbb{T}^2\to\mathbb{T}^2,\quad (n, (\omega,u), x)\mapsto \varphi_n(\omega,u)x.$$.
\end{lemma}
\begin{proof}
 Denote $(\mathcal{P}_t)_{t\in\mathbb{R}_+}$ as  the Markov semigroup  generated by $(u_t, \varphi_t)_{t\in\mathbb{R}_+}$.  According to the weak irreducibility (see \cite[Lemma 7.3]{MR4404792}) and strong Feller  property of $(\mathcal{P}_t)_{t\in\mathbb{R}_+}$,  one has that $\mu\times\text{vol}$ is also unique stationary measure for Markov semigroup $(\mathcal{P}_{mn})_{n\in\mathbb{Z}_+}$ for any $m\in\mathbb{N}$ by \cite[Corollary 3.17]{MR2478676}.  Therefore, $(\Omega\times\mathbb{H}\times\mathbb{T}^2, \mathbb{P}\times\mu\times\text{vol},(\mathscr{U}_{nm})_{n\in\mathbb{Z}_+})$    is ergodic  by combining \cite[Theorem 3.2.6]{MR1417491} and \cite[Theorem 2.1.7]{A}.  
 \end{proof}

\subsection{Integrable conditions for Lagrangian flow}

 In this subsection,  we mainly verify  the integral condition \eqref{23-7-5-2042}  for  the Lagrangian flow. Firstly,  we give two useful lemmas. 
 \begin{lemma}
 \label{23-8-7-1611}
Let $f\in \mathbf{H}^s(\mathbb{T}^2; \mathbb{R}^2)$ with $s> 3$.  Then $f\in C^2(\mathbb{T}^2;\mathbb{R}^2)$ and 
 $$|f|_{C^2}\lesssim_{s}\|f\|_{\mathbf{H}^s},$$
 where $a\lesssim_{s}b$ means that there exists a constant $C>0$ only depending on $s$ such that $a\leq Cb$.
 \end{lemma}
 This lemma follows from the classical Sobolev embedding theorem. The reader can refer to \cite[Theorem 3.13]{MR3494937} for a  simple case.
 \begin{lemma}[{\cite[Exercise 2.5.4]{MR3443633}}]
 \label{23-8-9-2031}
 Recall that $(u_t)_{t\in\mathbb{R}_+}$ is the  velocity flow   generated by \Cref{23-8-4-1509},  and  $\mu$ is  the unique stationary measure of the velocity flow   $(u_t)_{t\in\mathbb{R}_+}$. Then 
 $$\int_{\mathbb{H}}\|u\|_{\mathbf{H}^s}\mathrm{d}\mu<+\infty.$$
 \end{lemma}
\begin{proposition}
\label{23-8-9-2100}
Recall that $(\varphi_t)_{t\in\mathbb{R}_+}$ be the Lagrangian flow generated by \Cref{23-8-4-1509}. Then 
\begin{align}
\int_{\Omega\times\mathbb{H}}\big(\log^+\|\varphi_1(\omega,u)\|_{C^2}+\log^+\|\varphi_1(\omega,u)^{-1}\|_{C^2}\big)\mathrm{d}(\mathbb{P}\times\mu)<\infty.
\end{align}
\end{proposition}
\begin{proof}
Firstly,  recall a uniform estimation from \cite[Proposition A.3]{MR4404792} with a special case: for any $u\in \mathbb{H}$ one has that
\begin{align}
\int_{\Omega}\sup_{t\in[0,1]}\|u_t(\omega,u)\|_{\mathbf{H}^{s}}\mathrm{d}\mathbb{P}\lesssim_{s}1+\|u\|_{\mathbf{H}^{s}}.
\end{align}
 By utilizing  \Cref{23-8-7-1611},  we have that 
 \begin{align}
 \label{23-8-9-1444}
 \int_{\Omega}\sup_{t\in [0,1]}|u_t(\omega,u)|_{C^2}\mathrm{d}\mathbb{P}\lesssim_{s}1+\|u\|_{\mathbf{H}^{s}}.
 \end{align}
 According to \eqref{23-8-4-1509},     for any $x\in\mathbb{T}^2$ and  any $t>0$ one has that
 \begin{align}
 \label{23-8-9-1502}
 \varphi_t(\omega,u)x=x+\int_0^tu_{\tau}(\omega,u)(\varphi_{\tau}(\omega,u)x)\mathrm{d}\tau
 \end{align}
 holds for  $\mathbb{P}\times\mu$-a.s. $(\omega,u)\in\Omega\times\mathbb{H}$. Then, for any $i\in \{1,2\}$ one has that 
\begin{align}
|\partial_i\varphi_1(\omega,u)x|\leq 1+\int_0^1\|d_{{\varphi_t(\omega,u)x}}u_t(\omega,u)\|\cdot|\partial_i\varphi_t(\omega,u)x|\mathrm{d}t.
\end{align}
By using  Gronwal's inequality,  it is clear that $|\partial_i\varphi_1(\omega,u)x|\leq \exp(\int_0^1\|d_{\varphi_t(\omega,u)x}u_t(\omega,u)\|\mathrm{d}t)$. It follows that 
\begin{align*}
\int_{\Omega}\log^+\|d\varphi_1(\omega,u)\|\mathrm{d}\mathbb{P}&\leq\int_{\Omega}\int_0^1\sup_{x\in\mathbb{T}^2}\|d_{\varphi_t(\omega,u)x}u_t(\omega,u)\|\mathrm{d}t\mathrm{d}\mathbb{P}
\\&\overset{ \eqref{23-8-9-1444}} \lesssim_{s}1+\|u\|_{\mathbf{H}^{s}}.
\end{align*}
Taking  twice differential of  \eqref{23-8-9-1502} with respect to  space parameter,  for any $i,j\in\{1,2\}$ one has that 
\begin{align*}
|\partial_{ij}\varphi_1(\omega,u)|&\lesssim\int_0^1\|u_t(\omega,u)\|_{C^2}\big(|\partial_i\varphi_t(\omega,u)x|\cdot |\partial_j\varphi_t(\omega,u)x|+|\partial_{ij}\varphi_t{(\omega,u)}x|\big)\mathrm{d}t
\\&\lesssim \sup_{t\in[0,1]}\|u_t(\omega,u)\|_{C^2} \exp2(\int_0^1\|d_{\varphi_t(\omega,u)x}u_t(\omega,u)\|\mathrm{d}t)
\\&\quad+ \sup_{t\in[0,1]}\|u_t(\omega,u)\|_{C^2}\int_0^1|\partial_{ij}\varphi_t{(\omega,u)}x|\mathrm{d}t.
\end{align*}
By a similar argument as above, we have 
\begin{align*}
\int_{\Omega}\log^+ |d_2\varphi_1(\omega,u)|\mathrm{d}\mathbb{P}&\lesssim\int_{\Omega}\log\Big(\sup_{t\in[0,1]}\|u_t(\omega,u)\|_{C^2}+1\Big)+3\sup_{t\in[0,1]}\|u_t(\omega,u)\|_{C^2}\mathrm{d}\mathbb{P}
\\&\lesssim_{s}1+\|u\|_{\mathbf{H}^{s}}.
\end{align*}
Therefore,  together with  \Cref{23-8-9-2031} we have $\int_{\Omega\times\mathbb{H}}\log^+|\varphi_1(\omega,x)|_{C^2}\mathrm{d}(\mathbb{P}\times\mu)<+\infty.$  Finally, applying  the equality
$$\varphi_{t}(\omega,u)^{-1}x=x-\int_0^{t}u_{\tau}(\omega,u)(\varphi_{\tau}(\omega,u)\varphi_{t}^{-1}{(\omega,u)}x)\mathrm{d}\tau$$
for $t>0$ with above similar argument, we can obtain that   
$$\int_{\Omega\times\mathbb{H}}\log^+|\varphi_1(\omega,x)^{-1}|_{C^2}\mathrm{d}(\mathbb{P}\times\mu)<+\infty,$$
which implies the above proposition. 
\end{proof}

\subsection{Observable full-horseshoes}

To complete the proof of  \Cref{23-7-4-2042},  we recall a lemma to ensure  existence of  the full-horseshoes  which is from our previous work  \cite{axXiv: 230305027}. Although there is  a no  formal  result in \cite{axXiv: 230305027} as  following lemma,  the reader can begin with \cite[Lemma 4.5]{axXiv: 230305027}  to  obtain following lemma by using analogous arguments in \cite[Theorem 4.4  and Theorem 5.4]{axXiv: 230305027}.
\begin{lemma}
\label{23-7-8-2238}
Let $(\varphi,\mu)$ be an ergodic continuous\footnote{Here, $\varphi$ being a continuous RDS means that $\varphi_1(\omega)$ is a continuous map on $M$ for $\mathbb{P}$-a.s. $\omega\in\Omega$.} two-side measure-preserving  RDS on a compact metric space $M$ over a metric system $(\Omega,\mathscr{F},\mathbb{P}, (\theta_n)_{n\in\mathbb{Z}})$. If $h_{\mu}(\varphi)>0$, then  for any two disjoint non-empty closed balls $\{U_1, U_2\}$ of $M$ with 
 $$\mu\times_{\mathcal{P}_{\mu}}\mu\Big((\Omega\times U_1)\times (\Omega\times U_2)\Big)>0,$$
$\varphi$ has full-horseshoes on $\{U_1, U_2\}$, where $\mathcal{P}_{\mu}$ is the  Pinsker $\sigma$ partition of  measure-preserving RDS $(\varphi,\mu)$ and $\mu\times_{\mathcal{P}_{\mu}}\mu$ is the relative independent joining (see \cite[Chapter 6]{G}).
\end{lemma}
After completing all the preparatory work,  we begin to give the  proof of  \Cref{23-7-4-2042}.

\begin{proof}[Proof of  \Cref{23-7-4-2042}]
 Let 
 $$\hat{\pi}: (\hat{\Omega},\hat{\mathbb{P}}, (\hat{\theta}_n)_{n\in\mathbb{Z}})\to (\Omega\times\mathbb{H}, \mathbb{P}\times\mu, (\mathscr{U}_{n})_{n\in\mathbb{Z}_+})$$ be the invertible  extension system defined as \eqref{23-10-17-0002}, where $(\mathscr{U}_{n})_{n\in\mathbb{Z}_+}$ is the measurable maps generated by $(u_t)_{t\in\mathbb{R}_+}$  as  \eqref{23-8-15-1253},  and let $(\hat{\varphi},\hat{\mu})$  be the invertible extension of measure-preserving RDS $(\varphi^{(1)}, \mathbb{P}\times\mu\times\text{vol})$ on $\mathbb{T}^2$ over the metric system $(\Omega\times\mathbb{H},\mathbb{P}\times\mu, (\mathscr{U}_n)_{n\in\mathbb{Z}_+})$ (see \eqref{23-8-15-1413} for details).  It is clear to see that if   RDS  $\hat{\varphi}$ has observable full-horseshoes, then RDS $\varphi$ has observable RDS.  Therefore we only need to prove that RDS $\hat{\varphi}$ has observable RDS.

Firstly,  we verify that $(\hat{\varphi},\hat{\mu})$ is a $K$-system. By \Cref{23-8-9-2100}  and integral changing formula, one has that 
$$\int_{\hat{\Omega}}\big(\log^+\|\varphi_1(\hat{\omega})\|_{C^2}+\log^+\|\varphi_1(\hat{\omega})^{-1}\|_{C^2}\big)\mathrm{d}\hat{\mathbb{P}}(\hat{\omega})<+\infty.$$
By \cite[Theorem 1.6]{MR4404792},  we know that  $(\varphi^{(1)},\mathbb{P}\times\mu\times\text{vol})$ has positive Lyapunov exponents. Combining  \Cref{23-8-15-1431}, \Cref{23-8-15-1444} and \Cref{23-8-13-1921}, it is clear to see that $(\hat{\varphi},\hat{\mu})$ satisfies the all assumptions  in \Cref{23-6-27-1553}. Hence that the measure-preserving RDS $(\hat{\varphi},\hat{\mu})$ is a $K$-system.  It follows that 
  \begin{align}
  \label{23-9-3-13-17}
  \hat{\mu}\times_{\mathcal{P}_{\mu}}\hat{\mu}=\int_{\hat{\Omega}}\hat{\mu}_{\hat{\omega}}\times\hat{\mu}_{\hat{\omega}}\mathrm{d}\hat{\mathbb{P}}(\hat{\omega}),
  \end{align}
  where $\hat{\mu}=\int_{\hat{\Omega}}\hat{\mu}_{\hat{\omega}}\mathrm{d}\mathbb{P}$  is the disintegration  of the measure $\hat{\mu}$ with respect to measurable partition $\hat{\varpi}=\{\hat{\omega}\times\mathbb{T}^2: \hat{\omega}\in\hat{\Omega}\}$. Next, we would like to prove  following claim. 
  \begin{claim}
  \label{23-9-3-1320}
The measure $\hat{\mu}$ is full-support. Particularly, for $\hat{\mathbb{P}}$-a.s. $\hat{\omega}\in\hat{\Omega}$,  $\text{supp}(\hat{\mu}_{\hat{\omega}})=\mathbb{T}^2$. 
  \end{claim}
  \begin{proof}
Assume that  there exists a positive $\hat{\mathbb{P}}$-measure subset $\hat{\Omega}_1$ of $\hat{\Omega}$ such that $\text{supp}(\hat{\mu}_{\hat{\omega}})\neq\mathbb{T}^2$ for  any $\hat{\omega}\in\hat{\Omega}_1$.  By \cite[Proposition 1.6.11]{A} and ergodicity of $\hat{\mathbb{P}}$,  one has that $\hat{\mathbb{P}}(\hat{\Omega}_1)=1$.  Then there exists a random open set  $\hat{U}$ on $(\hat{\Omega},\hat{\mathscr{F}},\hat{\mathbb{P}})$ satisfying  $\hat{\mu}(\hat{U})=0$ (see \cite[Chapter 1.6]{A}), where random open set  $\hat{U}$ means  that 
\begin{enumerate} 
\item $\hat{U}$ is Borel measurable subset of $\hat{\Omega}\times\mathbb{T}^2$;
\item  $\hat{U}(\hat{\omega}):=\{x\in \mathbb{T}^2: (\omega, x)\in\hat{U}\}$ is an open subset of $\mathbb{T}^2$ for $\hat{\mathbb{P}}$-a.s. $\hat{\omega}\in\hat{\Omega}$.
\end{enumerate}  
By \cite[Lemma 4.3]{axXiv: 230305027}, there exists a measurable map 
$$\check{\pi}: (\Omega\times\mathbb{H},\mathscr{F}\otimes\mathscr{F}_{\mathbb{H}},\mathbb{P}\times\mu)\to(\hat{\Omega},\hat{\mathscr{F}},\hat{\mathbb{P}})  $$
such that $\hat{\pi}\circ\check{\pi}(\omega,u)=(\omega,u)$ for $\mathbb{P}\times\mu$-a.s. $(\omega,u)\in\Omega\times\mathbb{T}^2$.
Therefore, there exists a random open set $U$  on $(\Omega\times\mathbb{H},\mathscr{F}\otimes\mathscr{F}_{\mathbb{H}},\mathbb{P}\times\mu)$ satisfying that  $U(\omega,u)=\hat{U}\big(\check{\pi}(\omega,u ))\big)$ for $\mathbb{P}\times\mu$-a.s. $(\omega,u)\in\Omega\times\mathbb{H}$.  Then one has that 
\begin{align*}
\tag{By $(\hat{\pi}\times\text{Id}_{\mathbb{T}^2})_*\hat{\mu}=\mathbb{P}\times\mu\times\text{vol}$}0=\hat{\mu}(\hat{U})&=\mathbb{P}\times\mu\times\text{vol}(U)
\\&=\int_{\Omega\times\mathbb{H}}\text{vol}(U(\omega,u))\mathrm{d}(\mathbb{P}\times\mu)>0,
\end{align*}
which is contradictory.
\end{proof}
Finally, applying   \eqref{23-9-3-13-17} and   \Cref{23-9-3-1320}, one has that for any two disjoint non-empty closed balls $\{U_1, U_2\}$ of $\mathbb{T}^2$,
$$\hat{\mu}\times_{\mathcal{P}_{\hat{\mu}}}\hat{\mu}\Big((\hat{\Omega}\times U_1)\times (\hat{\Omega}\times U_2)\Big)=\int_{\hat{\Omega}}\hat{\mu}_{\hat{\omega}}(U_1)\hat{\mu}_{\hat{\omega}}(U_2)\mathrm{d}\hat{\mathbb{P}}(\hat{\omega})>0.$$
By \Cref{23-7-8-2238}, we know that the RDS  $\hat{\varphi}$ has full-horseshoes  on $\{U_1, U_2\}$.  Hence that   $\hat{\varphi}$ has observable full-horseshoes.  This finishes the proof of  \Cref{23-7-4-2042}.
\end{proof}

Through the proof,   one can obtain following abstract result (\Cref{23-8-19-2012}) to ensure the existence of the observable full-horseshoes for RDS.  It is notable that this result is even new when the RDS degenerates a deterministic dynamical system.  Deterministic systems under  powerful assumptions of  \Cref{23-8-19-2012}  have horseshoes (see \cite{MR0573822}),  but we can't get  precise  information about the locations of horseshoes.  The shining point of our result  is that  it  gives a kind of observable version of  chaotic structure (full-horseshoe).  
\begin{proposition}
\label{23-8-19-2012}
Let $(\varphi,\mu)$ be an ergodic two-side  $C^2$ measure-preserving  RDS on a closed Riemannian manifold $M$ over a metric system $(\Omega,\mathscr{F},\mathbb{P},(\theta_n)_{n\in\mathbb{Z}})$ satisfying \eqref{23-7-5-2042}.  If $\mu$ is hyperbolic, SRB, totally ergodic  and full-support, then $\varphi$ has observable full-horseshoes.
\end{proposition}

\providecommand{\bysame}{\leavevmode\hbox to3em{\hrulefill}\thinspace}
\providecommand{\MR}{\relax\ifhmode\unskip\space\fi MR }
\providecommand{\MRhref}[2]{%
  \href{http://www.ams.org/mathscinet-getitem?mr=#1}{#2}
}
\providecommand{\href}[2]{#2}


\begin{thebibliography}{10}


\bibitem{A}
L. Arnold, \emph{Random dynamical systems}, Springer Monographs in
  Mathematics, Springer-Verlag, Berlin, 1998.
  

  


\bibitem{MR3076414} 
L.~Barreira and Y.~Pesin, \emph{Introduction to smooth ergodic theory},  Graduate Studies in Mathematics, 148. American Mathematical Society, Providence, RI,  2013.


\bibitem{MR4404792}
 J.  Bedrossian, A. Blumenthal and S. Punshon-Smith, \emph{Lagrangian chaos and scalar advection in stochastic fluid mechanics},  J. Eur. Math. Soc.  \textbf{24}  (2022), no. 6, 1893--1990. 

 
  
  \bibitem{MR4372219}
J. Bedrossian, A. Blumenthal  and S. Punshon-Smith, \emph{A regularity
  method for lower bounds on the {L}yapunov exponent for stochastic
 differential equations}, Invent. Math. \textbf{227} (2022), no.~2, 429--516.




\bibitem{bedrossian2021chaos}
J. Bedrossian and S. Punshon-Smith, \emph{Chaos in stochastic 2d
  {G}alerkin-{N}avier-{S}tokes},  arXiv:2106.13748 (2021).



\bibitem{MR3925383}
A. Blumenthal, and L.-S. Young, \emph{Equivalence of physical and SRB measures in random dynamical systems}, Nonlinearity \textbf{32} (2019), no. 4, 1494--1524.











\bibitem{MR0805125}
A.~Carverhill, \emph{Flows of stochastic dynamical systems: ergodic theory},
Stochastics \textbf{14} (1985), no. 4, 273--317. 



\bibitem{CFVP}
A.~Crisanti, M.~ Falcioni, A.~ Vulpiani and G.~Paladin,  \emph{Lagrangian chaos: Transport, mixing and diffusion in fluids},  Riv. Nuovo Cim. \textbf{14} (1991), 1--80.


\bibitem{MR1417491}
G.~Da~Prato and J.~Zabczyk, \emph{Ergodicity for infinite-dimensional systems},
  London Mathematical Society Lecture Note Series, vol. 229, Cambridge
  University Press, Cambridge, 1996.






\bibitem{EW}
M. Einsiedler and T. Ward, \emph{Ergodic theory with a view towards
  number theory}, Graduate Texts in Mathematics, vol. 259, Springer-Verlag
  London, Ltd., London, 2011.


\bibitem{Zbl 1205.76133}
G.~Falkovich, K.~Gaw{\c{e}}dzki  and M.~Vergassola, \emph{Particles and fields in fluid turbulence}, Rev. Mod. Phys. \textbf{73} (2001), no.~4, 913--975.

\bibitem{Zbl 0796.76084}
A.~Fannjiang,  and G.~Papanicolaou, \emph{Convection enhanced diffusion for periodic flows}, SIAM J. Appl. Math. \textbf{54} (1994),  no.~2, 333--408.

\bibitem{MR1346374}
F.~Flandoli and B.~Maslowski, \emph{Ergodicity of the ${2}$-{D} Navier-Stokes equation under random perturbations},
Comm. Math. Phys. \textbf{172} (1995), no. 1, 119--141. 



\bibitem{G}
E. Glasner, \emph{Ergodic theory via joinings}, Mathematical Surveys and
  Monographs, vol. 101, American Mathematical Society, Providence, RI, 2003.

\bibitem{MR1786718}
E. Glasner,  J.-P. Thouvenot and B. Weiss, \emph{Entropy theory without a past},
Ergodic Theory Dynam. Systems \textbf{20} (2000), no. 5, 1355–1370. 
  

  
  \bibitem{GS}
J.P. Gollub and H.L. Swinney, \emph{Onset of turbulence in a rotating fluid},  Phys. Rev. Lett. \textbf{35} (1975), 927--930.









\bibitem{axXiv: 230305027}
W. Huang and J. Zhang, \emph{Full-horseshoes for the Galerkin truncations of 2D Navier-Stokes equation with degenerate stochastic forcing}, arXiv:2303.05027 (2023).





\bibitem{MR0573822}
 A.~Katok,  \emph{Lyapunov exponents, entropy and periodic orbits for diffeomorphisms},  Inst. Hautes \'Etudes Sci. Publ. Math, (1980),  no.~51, 137--173.






\bibitem{MR4264955}
V. Jak\v{s}i\'c,  V.~ Nersesyan,  C.-A.~ Pillet  and A.~ Shirikyan, \emph{Large deviations and entropy production in viscous fluid flows},
Arch. Ration. Mech. Anal.  \textbf{240} (2021), no. 3, 1675--1725. 


\bibitem{MR3443633}
S.~Kuksin, Sergei and A.~Shirikyan, \emph{Mathematics of two-dimensional turbulence},  Cambridge Tracts in Mathematics, 194. Cambridge University Press, Cambridge, 2012.  

\bibitem{KiL}
Y. Kifer and P.-D. Liu, \emph{Random dynamics}, Handbook of dynamical
  systems. {V}ol. 1{B}, Elsevier B. V., Amsterdam, 2006, pp.~379--499.

\bibitem{Zbl 1273.76083}
T.~Komorowski,   S.~Peszat  and T.~Szarek, \emph{Passive tracer in a flow corresponding to two-dimensional stochastic Navier-Stokes equations}, 
Nonlinearity \textbf{26} (2013),  no.~7, 1999--2026.

\bibitem{Zbl 0193.27106}
R.H.~Kraichnan,  \emph{Diffusion by a random velocity field}, Phys. Fluids \textbf{13} (1970), 22--31.




\bibitem{arXiv:2304.03685}
J.~ Lamb, G.~Tenaglia and D. Turaev, \emph{Horsehoes for a class of nonuniformly expanding random dynamical systems on the circle}, arXiv:2304.03685 (2023).




\bibitem{MR968818}
F.~Ledrappier and L.-S. Young, \emph{Entropy formula for random
  transformations}, Probab. Theory Related Fields \textbf{80} (1988), no.~2,
  217--240.


\bibitem{Lib}
A.~Libchaber, \emph{From chaos to turbulence in B\'enard convection,} Proc. Roy. Soc. London Ser. A  \textbf{413} (1987),  no.~1844, 63--69

\bibitem{MR1369243}
P.-D. Liu and M. Qian, \emph{Smooth ergodic theory of random dynamical
  systems}, Lecture Notes in Mathematics, vol. 1606, Springer-Verlag, Berlin,
  1995.
  
  
 \bibitem{MR1944407}
P.-D. Liu, M. Qian and F.-X.  Zhang, \emph{Entropy formula of Pesin type for one-sided stationary random maps}, Ergodic Theory Dynam. Systems, \textbf{22} (2002), no. 6, 1831--1844. 
 



\bibitem{MR3494937} 
G.~\L ukaszewicz and P.~ Kalita,
\emph{Navier-Stokes equations. An introduction with applications}, Advances in Mechanics and Mathematics, 34. Springer, Cham, 2016.  

%


\bibitem{majda2006nonlinear}
A. Majda and X. Wang, \emph{Nonlinear dynamics and statistical
  theories for basic geophysical flows}, Cambridge University Press, 2006.

\bibitem{MR2478676}
H. Martin  and  J.C. Mattingly, \emph{Ergodicity of the 2D Navier-Stokes equations with degenerate stochastic forcing}, 
Ann. of Math. (2) \textbf{164} (2006), no. 3, 993--1032. 





\bibitem{MR0262464}
W. Parry, \emph{Entropy and generators in ergodic theory},  W. A. Benjamin, Inc., New York-Amsterdam, 1969. 





\bibitem{MR0466791}
Ya. Pesin, \emph{Characteristic {L}yapunov exponents, and smooth ergodic
  theory}, Russ. Math. Surv.\textbf{32} (1977), no.~4, 55--114.


\bibitem{pope2000turbulent}
S. Pope, \emph{Turbulent flows}, Cambridge University Press, 2000.


\bibitem{MR0143873}
V. Rohlin, \emph{Exact endomorphisms of a {L}ebesgue space}, Izv. Akad. Nauk
 SSSR Ser. Mat. \textbf{25} (1961), 499--530.

\bibitem{MR0217258}
V. Rohlin, \emph{Lectures on the entropy theory of transformations with invariant
  measure}, Uspehi Mat. Nauk \textbf{22} (1967), no.~5 (137), 3--56.

\bibitem{MR2770903}
M.~Romito and L.~ Xu, \emph{Ergodicity of the ${3D}$ stochastic Navier-Stokes equations driven by mildly degenerate noise}, Stochastic Process. Appl. \textbf{121} (2011), no. 4, 673--700. 


\bibitem{MR0284067} 
D.~Ruelle and F. ~Takens, \emph{On the nature of turbulence},  Comm. Math. Phys. \textbf{20} (1971), 167--192. 





\bibitem{salmon1998lectures}
R. Salmon, \emph{Lectures on geophysical fluid dynamics}, Oxford University
  Press, New York, 1998.




\bibitem{MR0182020}
S. Smale, \emph{Diffeomorphisms with many periodic points}, Differential
  and {C}ombinatorial {T}opology ({A} {S}ymposium in {H}onor of {M}arston
  {M}orse), Princeton Univ. Press, Princeton, N.J., 1965, pp.~63--80.


\bibitem{MR2140094}
M.~Shub, \emph{What is  $\cdots$ a horseshoe?} Notices Amer. Math. Soc. \textbf{52} (2005), no. 5, 516--517.



\bibitem{vallis2017atmospheric}
G. Vallis, \emph{Atmospheric and oceanic fluid dynamics. Fundamentals and large-scale circulation}, Cambridge
  University Press, 2017.




\bibitem{Wal}
P. Walters, \emph{An introduction to ergodic theory}, Graduate Texts in
  Mathematics, vol.~79, Springer-Verlag, New York-Berlin, 1982.

\bibitem{MR2197376}
G. Zhang, \emph{Relative entropy, asymptotic pairs and chaos}, J. London
  Math. Soc. (2) \textbf{73} (2006), no.~1, 157--172.


\end{thebibliography}
\end{document}